\documentclass[leqno,12pt]{article} 
\usepackage{amsmath, amssymb}
\usepackage{amsthm} 
\theoremstyle{plain} 
\newtheorem{theorem}{\indent\sc Theorem}[section]

\newtheorem{proposition}[theorem]{\indent\sc Proposition}

\theoremstyle{definition} 

\newtheorem{remark}[theorem]{\indent\sc Remark}
\newtheorem{example}[theorem]{\indent\sc Example}

\title{Solitons and geometrical structures \\in a perfect fluid spacetime} 
\author{Adara M. Blaga}

\date{}
\begin{document}

\maketitle

\markboth{{\small\it {\hspace{4cm} Solitons in a perfect fluid spacetime}}}{\small\it{Solitons in a perfect fluid spacetime
\hspace{4cm}}}


\begin{abstract}
Geometrical aspects of a perfect fluid spacetime are described in terms of different curvature tensors and $\eta$-Ricci and $\eta$-Einstein solitons in a perfect fluid spacetime are determined. Conditions for the Ricci soliton to be steady, expanding or shrinking are also given. In a particular case when the potential vector field $\xi$ of the soliton is of gradient type, $\xi:=grad(f)$, we derive a Poisson equation from the soliton equation.
\end{abstract}

\section{Introduction}

Lorentzian manifolds form a special subclass of pseudo-Riemannian manifolds of great importance in general relativity, where spacetime can be modeled as a $4$-dimensional Lorentzian manifold of signature $(3,1)$ or, equivalently, $(1,3)$.

Relativistic fluid models are of great interest in different branches of astrophysics, plasma physics, nuclear physics, etc. Perfect fluids are often used in general relativity to model idealized distributions of matter, such as the interior of a star or an isotropic universe. Einstein's gravitational equation can describe the behavior of a perfect fluid inside a spherical object and the Friedmann-Lema\^{\i}tre-Robertson-Walker equations are used to describe the evolution of the universe. In general relativity, the source for the gravitational field is the energy-momentum tensor.
A perfect fluid can be completely characterized by its rest frame mass density and isotropic pressure. It has no shear stresses, viscosity, nor heat conduction, and it is characterized by an energy-momentum tensor of the form:
\begin{equation}\label{e3}
T(X,Y)=pg(X,Y)+(\sigma+p)\eta(X)\eta(Y),
\end{equation}
for any $X$, $Y\in\chi(M)$, where $p$ is the isotropic pressure, $\sigma$ is the energy-density, $g$ is the metric tensor of Minkowski spacetime, $\xi:=\sharp(\eta)$ is the velocity vector of the fluid and $g(\xi,\xi)=-1$. If $\sigma=-p$,
the energy-momentum tensor is Lorentz-invariant ($T=-\sigma g$) and in this case we talk about the \textit{vacuum}. If $\sigma=3p$, the medium is a \textit{radiation fluid}.

The field equations governing the perfect fluid motion are Einstein's gravitational equations:
\begin{equation}\label{e2}
kT(X,Y)=S(X,Y)+\left(\lambda-\frac{r}{2}\right)g(X,Y),
\end{equation}
for any $X$, $Y\in\chi(M)$, where $\lambda$ is the cosmological constant,
$k$ is the gravitational constant (which can be taken $8\pi G$, with $G$ the universal gravitational constant), $S$ is the Ricci tensor and $r$ is the scalar curvature of $g$. They are obtained from Einstein's equations by adding a cosmological constant in order to get a static universe, according to Einstein's idea. In modern cosmology, it is considered as a candidate for dark energy, the cause of the accelerated expansion of the universe.

Replacing $T$ from (\ref{e3}) we obtain:
\begin{equation}\label{e5}
S(X,Y)=-\left(\lambda-\frac{r}{2}-kp\right)g(X,Y)+k(\sigma+p)\eta(X)\eta(Y),
\end{equation}
for any $X$, $Y\in\chi(M)$. Recall that a manifold having the property that the Ricci tensor $S$ is a functional combination of $g$ and $\eta\otimes \eta$, for $\eta$ the $g$ dual $1$-form of a unitary vector field, is called \textit{quasi-Einstein} \cite{mai}. Quasi-Einstein manifolds arose during the study of exact solutions of Einstein's field equations. For example, Robertson-Walker spacetimes are quasi-Einstein manifolds \cite{de}. They also can be taken as a
model of the perfect fluid spacetime in general relativity \cite{degh}, \cite{degho}.

\textit{Ricci flow} and \textit{Einstein flow} are intrinsic geometric flows on a pseudo-Riemannian manifold, whose fixed points are \textit{solitons}. In our paper, we are interested in a generalized version of the following two types of solitons:
\begin{enumerate}
  \item \textit{Ricci solitons} \cite{ham}, which generates self-similar solutions to the Ricci flow:
\begin{equation}
\frac{\partial }{\partial t}g=-2S,
\end{equation}
  \item \textit{Einstein solitons} \cite{ca}, which generate self-similar solutions to the Einstein flow:
\begin{equation}
\frac{\partial }{\partial t}g=-2\left(S-\frac{r}{2}g\right).
\end{equation}
\end{enumerate}

Perturbing the equations that define these types of solitons by a multiple of a certain $(0,2)$-tensor field $\eta\otimes \eta$, we obtain two slightly more general notions, namely, \textit{$\eta$-Ricci solitons} and \textit{$\eta$-Einstein solitons}, which we shall consider in a perfect fluid spacetime, i.e. in a $4$-dimensional pseudo-Riemannian manifold $M$ with a Lorentzian metric $g$, whose content is a perfect fluid.

\section{Basic properties of a perfect fluid spacetime}

Let $(M,g)$ be a general relativistic perfect fluid spacetime satisfying (\ref{e5}).
Consider $\{E_i\}_{1\leq i \leq 4}$ an orthonormal frame field, i.e. $g(E_i,E_j)=\varepsilon_{ij}\delta_{ij}$, $i$, $j\in\{1,2,3,4\}$ with $\varepsilon_{11}=-1$, $\varepsilon_{ii}=1$, $i\in\{2,3,4\}$, $\varepsilon_{ij}=0$, $i,j\in\{1,2,3,4\}$, $i\neq j$. Let $\xi=\sum_{i=1}^{4}\xi^iE_i$. Then
$$-1=g(\xi,\xi)=\sum_{1\leq i,j\leq 4}\xi^i\xi^jg(E_i,E_j)=\sum_{i=1}^4\varepsilon_{ii}(\xi^i)^2$$
and
$$\eta(E_i)=g(E_i,\xi)=\sum_{j=1}^{4}\xi^jg(E_i,E_j)=\varepsilon_{ii}\xi^i.$$

Contracting (\ref{e5}) and taking into account that $g(\xi,\xi)=-1$, we get:
\begin{equation}\label{e8}
r=4\lambda+k(\sigma-3p).
\end{equation}

Therefore:
\begin{equation}\label{e52}
S(X,Y)=\left(\lambda+\frac{k(\sigma-p)}{2}\right)g(X,Y)+k(\sigma+p)\eta(X)\eta(Y),
\end{equation}
for any $X$, $Y\in\chi(M)$.

\begin{example}
A radiation fluid has constant scalar curvature equal to $4\lambda$.
\end{example}

Assume now that $p$ and $\sigma$ are constant.
If we denote by $\nabla$ the Levi-Civita connection associated to $g$, then for any $X$, $Y$, $Z\in\chi(M)$:
\begin{equation}\label{e53}
(\nabla_XS)(Y,Z):=X(S(Y,Z))-S(\nabla_XY,Z)-S(Y,\nabla_XZ)=$$
$$=k(\sigma+p)[\eta(Y)g(\nabla_X\xi,Z)+\eta(Z)g(\nabla_X\xi,Y)]=$$
$$=k(\sigma+p)[\eta(Y)(\nabla_X\eta)Z+\eta(Z)(\nabla_X\eta)Y].
\end{equation}

Imposing different conditions on the covariant differential of $S$, we have that:

i) $M$ is \textit{Ricci symmetric} if $\nabla S=0$;

ii) $S$ is a \textit{Codazzi tensor} if $(\nabla_X S)(Y,Z)=(\nabla_Y S)(X,Z)$, for any $X$, $Y$, $Z\in \chi(M)$;

iii) $S$ is \textit{$\alpha$-recurrent} if $(\nabla_XS)(Y,Z)=\alpha(X)S(Y,Z)$, for any $X$, $Y$, $Z\in \chi(M)$, with $\alpha$ a nonzero $1$-form;

iv) $S$ is \textit{weakly pseudo Ricci symmetric} if $S$ is not identically zero and $(\nabla_XS)(Y,Z)=\alpha(X)S(Y,Z)+\alpha(Y)S(Z,X)+\alpha(Z)S(X,Y)$, for any $X$, $Y$, $Z\in \chi(M)$, with $\alpha$ a nonzero $1$-form;

v) $S$ is \textit{pseudo Ricci symmetric} if $S$ is not identically zero and $(\nabla_XS)(Y,Z)=2\alpha(X)S(Y,Z)+\alpha(Y)S(Z,X)+\alpha(Z)S(X,Y)$, for any $X$, $Y$, $Z\in \chi(M)$, with $\alpha$ a nonzero $1$-form.

\begin{proposition}\label{p}
Let $(M,g)$ be a general relativistic perfect fluid spacetime satisfying (\ref{e52}) with $p$ and $\sigma$ constant.
\begin{enumerate}
  \item If $M$ is Ricci symmetric or $S$ is a Codazzi tensor, then $p=-\sigma$ or $\nabla \xi=0$.
   \item If $S$ is $\alpha$-recurrent, then $p=-\sigma=\frac{\lambda}{k}$ or $\nabla \xi=0$.
   \item If $S$ is (weakly) pseudo Ricci symmetric, then $p=\frac{2}{3}(\frac{\lambda}{k})-\frac{\sigma}{3}$. In this cases, $\xi$ is torse-forming (in particular, irrotational and geodesic) vector field and $\eta$ is a closed (and Codazzi) $1$-form.
  \end{enumerate}
\end{proposition}
\begin{proof}
\begin{enumerate}
  \item If $\nabla S=0$, writing (\ref{e53}) for $Y=Z$ we get $(\sigma+p)\eta(Y)g(\nabla_X\xi,Y)=0$, for any $X$, $Y\in \chi(M)$. It follows $\nabla\xi=0$ or $p=-\sigma$.

If $(\nabla_X S)(Y,Z)=(\nabla_Y S)(X,Z)$, for any $X$, $Y$, $Z\in \chi(M)$, using (\ref{e53}) for $Y=Z:=\xi$ we get $(\sigma+p)g(\nabla_{\xi}\xi,X)=0$, for any $X\in \chi(M)$. It follows $\nabla_{\xi}\xi=0$ or $p=-\sigma$. If $\nabla_{\xi}\xi=0$, using (\ref{e53}) for $X:=\xi$ we get $(\sigma+p)g(\nabla_Y\xi,Z)=0$, for any $Y$, $Z\in \chi(M)$. It follows $\nabla\xi=0$ or $p=-\sigma$.
 \item If $(\nabla_XS)(Y,Z)=\alpha(X)S(Y,Z)$, for any $X$, $Y$, $Z\in \chi(M)$, using (\ref{e52}) and (\ref{e53}) and writing the obtained relation for $Y=Z:=\xi$ we get
      $$\left(\lambda-\frac{k(\sigma+3p)}{2}\right)\alpha(X)=0,$$ for any $X\in \chi(M)$. It follows $p=\frac{2\lambda-k\sigma}{3k}$. Writing the same relation for $Z:=\xi$ we get
      $$-k(\sigma+p)g(\nabla_X\xi,Y)=\left(\lambda-\frac{k(\sigma+3p)}{2}\right)\alpha(X)\eta(Y)=0,$$ for any $X$, $Y\in \chi(M)$. It follows $\nabla \xi=0$ or $p=-\sigma$ (so $p=-\sigma=\frac{\lambda}{k}$).
      \item If $(\nabla_XS)(Y,Z)=\alpha(X)S(Y,Z)+\alpha(Y)S(Z,X)+\alpha(Z)S(X,Y)$, for any $X$, $Y$, $Z\in \chi(M)$, using (\ref{e52}) and (\ref{e53}) and writing the obtained relation for $Y=Z:=\xi$ we get
      $$\left(\lambda-\frac{k(\sigma+3p)}{2}\right)[\alpha(X)-2\alpha(\xi)\eta(X)]=0,$$ for any $X\in \chi(M)$. Writing the same relation for $X=Y=Z:=\xi$ we get $$\left(\lambda-\frac{k(\sigma+3p)}{2}\right)\alpha(\xi)=0.$$
      If $\lambda-\frac{k(\sigma+3p)}{2}\neq 0$ follows $\alpha(\xi)=0$ and $\alpha(X)=2\alpha(\xi)\eta(X)=0$, for any $X\in \chi(M)$, which contradicts the fact that $\alpha$ is nonzero. Therefore, $\lambda-\frac{k(\sigma+3p)}{2}=0$, so $p=\frac{2\lambda-k\sigma}{3k}$.

      Writing now the obtained relation for $Z:=\xi$ we get
      $$k(\sigma+p)\{g(\nabla_X\xi,Y)+\alpha(\xi)[g(X,Y)+\eta(X)\eta(Y)]\}=0,$$ for any $X$, $Y\in \chi(M)$.

      If $\sigma+p=0$ follows $p=-\sigma=\frac{\lambda}{k}$ and from (\ref{e52}), $S(X,Y)=0$, for any $X$, $Y\in \chi(M)$, which contradicts the fact that $S$ is nonzero. Therefore: $$g(\nabla_X\xi,Y)+\alpha(\xi)[g(X,Y)+\eta(X)\eta(Y)]=0,$$ for any $X$, $Y\in \chi(M)$ which implies
      $$\nabla_X\xi=-\alpha(\xi)[X+\eta(X)\xi],$$ for any $X\in \chi(M)$, i.e. $\xi$ is a torse-forming vector field. Then $$(curl(\xi))(X,Y):=g(\nabla_X\xi,Y)-g(\nabla_Y\xi,X)=0,$$ for any $X$, $Y\in \chi(M)$, i.e. $\xi$ is irrotational. In particular $$g(\nabla_{\xi}\xi,X)=g(\nabla_X\xi,\xi)=\frac{1}{2}X(g(\xi,\xi))=0,$$ for any $X\in \chi(M)$, i.e. $\xi$ is a geodesic vector field.

      Concerning $\eta$, notice that $$(\nabla_X\eta)Y-(\nabla_Y\eta)X:=X(\eta(Y))-\eta(\nabla_XY)-Y(\eta(X))+\eta(\nabla_YX):=$$$$:=(d\eta)(X,Y).$$

      Also
      $$X(\eta(Y))-\eta(\nabla_XY)-Y(\eta(X))+\eta(\nabla_YX)=g(\nabla_X\xi,Y)-g(\nabla_Y\xi,X)=0$$
      which implies $d\eta=0$.

      If $(\nabla_XS)(Y,Z)=2\alpha(X)S(Y,Z)+\alpha(Y)S(Z,X)+\alpha(Z)S(X,Y)$, for any $X$, $Y$, $Z\in \chi(M)$, following the steps of the computations above, we get the same conclusions.
\end{enumerate}
\end{proof}

Remark the following facts:

\smallskip

1) If in the general relativistic perfect fluid spacetime the vector field $\xi$ is not $\nabla$-parallel and the Ricci tensor field $S$ is $\alpha$-recurrent, we have the vacuum case. The energy-momentum tensor is $T_{vac}=\frac{\lambda}{k}g$.

\smallskip

2) $S(X,\xi)=-S(\xi,\xi)\eta(X)=\left[\lambda-\frac{k(\sigma+3p)}{2}\right]g(X,\xi)$, for any $X\in \chi(M)$ which shows that $\lambda-\frac{k(\sigma+3p)}{2}$ is the eigenvalue of $Q$ corresponding to the eigenvector $\xi$, where $Q$ is defined by $g(QX,Y):=S(X,Y)$, $X$, $Y\in \chi(M)$. From Proposition \ref{p} we deduce that if $\xi$ is not $\nabla$-parallel and $M$ is Ricci symmetric or $S$ is a Codazzi tensor, then $(M,g)$ is Einstein and $S(X,\xi)=(\lambda+k\sigma)\eta(X)$, for any $X\in \chi(M)$, and if $S$ is (weakly) pseudo Ricci symmetric, then $S(X,\xi)=0$, for any $X\in \chi(M)$, hence $\xi\in \ker Q$; also, $div(\xi):=\sum_{i=1}^4\varepsilon_{ii}g(\nabla_{E_i}\xi,E_i)=-3\alpha(\xi)$. Notice that in all these cases, if $\sigma>-\frac{\lambda}{k}$, the scalar curvature $r:=\sum_{i=1}^4\varepsilon_{ii}S(E_i,E_i)=4\lambda+k(\sigma-3p)>\lambda-kp$ is positive.

\smallskip

3) Considering Plebanski energy conditions $\sigma\geq 0$ and $-\sigma\leq p \leq \sigma$ for perfect fluids, when $S$ is (weakly) pseudo Ricci symmetric, the energy-density is lower bounded by $\max \{-\frac{\lambda}{k},\frac{\lambda}{2k}\}$. It was observed that a positive cosmological constant $\lambda$ acts as repulsive gravity, explaining the accelerating universe. The observations of Edwin Hubble confirmed that the universe is expanding, therefore, this seems to be the real case, so $\sigma \geq \frac{\lambda}{2k}$.

\smallskip

4) Let $hX:=X+\eta(X)\xi$ be the projection tensor, $X\in \chi(M)$. If $S$ is (weakly) pseudo Ricci symmetric and $\alpha(\xi)\neq 0$, then $(h,\xi,\eta,g)$ is a \textit{Lorentzian concircular structure} \cite{shaik}. In the particular case $\alpha(\xi)=-1$ it becomes \textit{LP-Sasakian structure} \cite{matsu}. In the other case, if $\alpha(\xi)=0$, then $\xi$ and $\eta$ are $\nabla$-parallel, $\xi$ is divergence-free, hence harmonic vector field.

\bigskip

Assume now that $p$ and $\sigma$ are not necessarily constant. Applying the covariant derivative to (\ref{e2}) we obtain:
\begin{equation}
k(\nabla_XT)(Y,Z)=(\nabla_XS)(Y,Z)-\frac{1}{2}X(r)g(Y,Z)=$$$$=k\{X(p)g(Y,Z)+
X(\sigma+p)\eta(Y)\eta(Z)+$$$$+(\sigma+p)[\eta(Y)g(\nabla_X\xi,Z)+\eta(Z)g(\nabla_X\xi,Y)]\},
\end{equation}
for any $X$, $Y$, $Z\in\chi(M)$.

\begin{proposition}
Let $(M,g)$ be a general relativistic perfect fluid spacetime satisfying (\ref{e52}).
\begin{enumerate}
  \item If $T$ is covariantly constant, then the energy-density is constant.
  \item If $T$ is a Codazzi tensor, then $d\sigma=-\xi(\sigma)\eta-(\sigma+p)\flat({\nabla_{\xi}\xi})$.
  \item If $T$ is $\alpha$-recurrent, then $d\sigma=\sigma\alpha$.
\end{enumerate}
\end{proposition}
\begin{proof}
\begin{enumerate}
  \item If $T$ is covariantly constant, i.e. $(\nabla_X T)(Y,Z)=0$, for any $X$, $Y$, $Z\in\chi(M)$, taking $Y=Z:=\xi$ we get $X(\sigma)=0$, for any $X\in\chi(M)$.
Also from (\ref{e8}) we get $dr=-3kdp$.
  \item The condition $(\nabla_XT)(Y,Z)=(\nabla_YT)(X,Z)$, for any $X$, $Y$, $Z\in\chi(M)$ is equivalent to:
$$X(p)g(Y,Z)+
X(\sigma+p)\eta(Y)\eta(Z)+$$$$+(\sigma+p)[\eta(Y)g(\nabla_X\xi,Z)+\eta(Z)g(\nabla_X\xi,Y)]=$$$$=
Y(p)g(X,Z)+
Y(\sigma+p)\eta(X)\eta(Z)+$$$$+(\sigma+p)[\eta(X)g(\nabla_Y\xi,Z)+\eta(Z)g(\nabla_Y\xi,X)],$$
for any $X$, $Y$, $Z\in\chi(M)$.
For $Y=Z:=\xi$ and taking into account that $g(\nabla_X\xi,\xi)=0$, for any $X\in\chi(M)$, the above relation becomes:
$$X(\sigma)=-\xi(\sigma)\eta(X)-(\sigma+p)g(\nabla_{\xi}\xi,X),$$
for any $X\in\chi(M)$.
  \item The condition $(\nabla_XT)(Y,Z)=\alpha(X)T(Y,Z)$, for any $X$, $Y$, $Z\in\chi(M)$, with $\alpha$ a nonzero $1$-form, is equivalent to:
$$X(p)g(Y,Z)+
X(\sigma+p)\eta(Y)\eta(Z)+$$$$+(\sigma+p)[\eta(Y)g(\nabla_X\xi,Z)+\eta(Z)g(\nabla_X\xi,Y)]=$$$$=
\alpha(X)[pg(Y,Z)+(\sigma+p)\eta(Y)\eta(Z)],$$
for any $X$, $Y$, $Z\in\chi(M)$. For $Y=Z:=\xi$ and taking into account that $g(\nabla_X\xi,\xi)=0$, for any $X\in\chi(M)$, the above relation becomes:
$$X(\sigma)=\sigma\alpha(X),$$
for any $X\in\chi(M)$. Also from (\ref{e8}) we get $dr=k(\sigma \alpha-3dp)$.
\end{enumerate}
\end{proof}

Remark the following facts:

\smallskip

If $T$ is a Codazzi tensor and $\xi$ is a geodesic vector field or the fluid is the vacuum, then the gradient of the energy-density is collinear with $\xi$.
If $T$ is a Codazzi tensor and the energy-density is constant, then we have the vacuum case or $\xi$ is a geodesic vector field.
It was proved that for a Codazzi energy-momentum tensor, the Ricci tensor $S$ is conserved \cite{ah}.

\section{Perfect fluid spacetime with torse-forming vector field $\xi$}

We shall treat the special case when $\xi$ is a torse-forming vector field \cite{ya} of the form:
\begin{equation}
\nabla \xi=I_{\chi(M)}+\eta\otimes \xi.
\end{equation}

Then $\nabla_{\xi}\xi=\xi+\eta(\xi)\xi=0$, i.e. $\xi$ is a geodesic vector field, and we have $g(\nabla_X\xi,\xi)=0$ and $(d\eta)(X,Y):=X(\eta(Y))-Y(\eta(X))-\eta([X,Y])=X(g(Y,\xi))-Y(g(X,\xi))-g(\nabla_XY,\xi)+g(\nabla_YX,\xi)=
g(\nabla_X\xi,Y)-g(\nabla_Y\xi,X)=0$. We also get:
\begin{equation}\label{e23}
R(X,Y)\xi=\eta(Y)X-\eta(X)Y,
\end{equation}
\begin{equation}\label{e24}
\eta(R(X,Y)Z)=-\eta(Y)g(X,Z)+\eta(X)g(Y,Z),
\end{equation}
for any $X$, $Y$, $Z\in \chi(M)$.

In this case, we shall see which are the consequences of certain conditions imposed to different types of curvatures of this space, namely, when the curvature satisfies conditions of the type $(\xi,\cdot)_{\mathcal{T}}\cdot S=0$ and $(\xi,\cdot)_{S}\cdot \mathcal{T}=0$, where
$\mathcal{T}$ stands for the \textit{Riemann curvature tensor} $R$,
the \textit{concircular curvature tensor} $\mathcal{P}$ \cite{yan}, the \textit{conformal curvature tensor} $\mathcal{C}$ \cite{ad} and the \textit{conharmonic curvature tensor} $\mathcal{H}$ \cite{po} defined as follows:
\begin{equation}\label{e28}
\mathcal{P}(X,Y)Z:=R(X,Y)Z+\frac{r}{\dim(M)(\dim(M)-1)}[g(Z,X)Y-g(Y,Z)X]
\end{equation}
\begin{equation}\label{e31}
\mathcal{C}(X,Y)Z:=R(X,Y)Z+
$$
$$+\frac{1}{\dim(M)-2}\{-\frac{r}{\dim(M)-1}[g(Z,X)Y-g(Y,Z)X]+$$$$+g(Z,X)QY-g(Y,Z)QX+S(Z,X)Y-S(Y,Z)X\}
\end{equation}
\begin{equation}\label{e32}
\mathcal{H}(X,Y)Z:=R(X,Y)Z+$$
$$+\frac{1}{\dim(M)-2}[g(Z,X)QY-g(Y,Z)QX+S(Z,X)Y-S(Y,Z)X]
\end{equation}
for any $X$, $Y$, $Z\in \chi(M)$.

Remark that for an empty gravitational field characterized by vanishing Ricci tensor, the curvature tensors $R$, $\mathcal{P}$, $\mathcal{C}$ and $\mathcal{H}$ coincide.

Let us denote by $\mathcal{T}$ a curvature tensor of type $(1,3)$ and ask for certain Ricci-semisymmetry curvature conditions, namely, $(\xi,\cdot)_{\mathcal{T}}\cdot S=0$ and $(\xi,\cdot)_{S}\cdot \mathcal{T}=0$, where by $\cdot$ we denote the derivation of the tensor algebra at each point of the tangent space:
\begin{itemize}
  \item $((\xi,X)_{\mathcal{T}}\cdot S)(Y,Z):=((\xi\wedge_\mathcal{T}X)\cdot S)(Y,Z):=S((\xi\wedge_{\mathcal{T}}X)Y,Z)+S(Y,(\xi\wedge_{\mathcal{T}}X)Z)$, for $(X\wedge_{\mathcal{T}}Y)Z:=\mathcal{T}(X,Y)Z$;
  \item $((\xi,X)_{S}\cdot \mathcal{T})(Y,Z)W:=(\xi\wedge_SX)\mathcal{T}(Y,Z)W+\mathcal{T}((\xi\wedge_SX)Y,Z)W+\linebreak \mathcal{T}(Y,(\xi\wedge_SX)Z)W+\mathcal{T}(Y,Z)(\xi\wedge_SX)W$, for $(X\wedge_SY)Z:=S(Y,Z)X-S(X,Z)Y$.
\end{itemize}

\subsection{Perfect fluid spacetime satisfying $(\xi,\cdot)_{\mathcal{T}}\cdot S=0$}

The condition is equivalent to
\begin{equation}
S(\mathcal{T}(\xi,X)Y,Z)+S(Y,\mathcal{T}(\xi,X)Z)=0,
\end{equation}
for any $X$, $Y$, $Z\in \chi(M)$ and from (\ref{e52}) we get:
$$\left(\lambda+\frac{k(\sigma-p)}{2}\right)[g(\mathcal{T}(\xi,X)Y,Z)+g(Y,\mathcal{T}(\xi,X)Z)]+$$$$+k(\sigma+p)[\eta(\mathcal{T}(\xi,X)Y)\eta(Z)+
\eta(Y)\eta(\mathcal{T}(\xi,X)Z)]=0,$$
for any $X$, $Y$, $Z\in \chi(M)$.

\begin{theorem}
Let $(M,g)$ be a general relativistic perfect fluid spacetime satisfying (\ref{e5}) with torse-forming vector field $\xi$.
\begin{enumerate}
  \item If $(\xi,\cdot)_{R}\cdot S=0$, then $p=-\sigma$.
  \item If $(\xi,\cdot)_{\mathcal{P}}\cdot S=0$, then $p=-\sigma$ or $p=\frac{4\lambda+k\sigma-12}{3k}$.
  \item If $(\xi,\cdot)_{\mathcal{C}}\cdot S=0$, then $p=-\sigma$ or $p=\frac{2\lambda-k\sigma-6}{3k}$.
  \item If $(\xi,\cdot)_{\mathcal{H}}\cdot S=0$, then $p=-\sigma$ or $p=\frac{\lambda-1}{k}$.
\end{enumerate}
\end{theorem}
\begin{proof}
\begin{enumerate}
  \item From the symmetries of $R$ and (\ref{e23}) and (\ref{e24}) we obtain:
\begin{equation}\label{e21}
k(\sigma+p)[\eta(Z)g(X,Y)+\eta(Y)g(Z,X)+2\eta(X)\eta(Y)\eta(Z)]=0,
\end{equation}
for any $X$, $Y$, $Z\in \chi(M)$. Take $Z:=\xi$ and (\ref{e21}) becomes:
\begin{equation}\label{e22}
k(\sigma+p)[g(X,Y)+\eta(X)\eta(Y)]=0,
\end{equation}
for any $X$, $Y\in \chi(M)$ and we obtain $p=-\sigma$.
  \item From (\ref{e28}), from the symmetries of $R$ and (\ref{e23}) and (\ref{e24}) we obtain:
\begin{equation}\label{e29}
k(\sigma+p)[4\lambda+k(\sigma-3p)-12][\eta(Z)g(X,Y)+\eta(Y)g(Z,X)+2\eta(X)\eta(Y)\eta(Z)]=0,
\end{equation}
for any $X$, $Y$, $Z\in \chi(M)$. Take $Z:=\xi$ and (\ref{e29}) becomes:
\begin{equation}\label{e30}
k(\sigma+p)[4\lambda+k(\sigma-3p)-12][g(X,Y)+\eta(X)\eta(Y)]=0,
\end{equation}
for any $X$, $Y\in \chi(M)$ and we obtain $p=-\sigma$ or $p=\frac{4\lambda+k\sigma-12}{3k}$.
  \item From (\ref{e31}), from the symmetries of $R$ and (\ref{e23}) and (\ref{e24}) we obtain:
\begin{equation}\label{...}
k(\sigma+p)[2\lambda-k(\sigma+3p)-6][\eta(Z)g(X,Y)+\eta(Y)g(Z,X)+2\eta(X)\eta(Y)\eta(Z)]=0,
\end{equation}
for any $X$, $Y$, $Z\in \chi(M)$. Take $Z:=\xi$ and (\ref{...}) becomes:
\begin{equation}\label{...}
k(\sigma+p)[2\lambda-k(\sigma+3p)-6][g(X,Y)+\eta(X)\eta(Y)]=0,
\end{equation}
for any $X$, $Y\in \chi(M)$ and we obtain $p=-\sigma$ or $p=\frac{2\lambda-k\sigma-6}{3k}$.
  \item From (\ref{e32}), from the symmetries of $R$ and (\ref{e23}) and (\ref{e24}) we obtain:
\begin{equation}\label{...}
k(\sigma+p)(\lambda-kp-1)[\eta(Z)g(X,Y)+\eta(Y)g(Z,X)+2\eta(X)\eta(Y)\eta(Z)]=0,
\end{equation}
for any $X$, $Y$, $Z\in \chi(M)$. Take $Z:=\xi$ and (\ref{...}) becomes:
\begin{equation}\label{...}
k(\sigma+p)(\lambda-kp-1)[g(X,Y)+\eta(X)\eta(Y)]=0,
\end{equation}
for any $X$, $Y\in \chi(M)$ and we obtain $p=-\sigma$ or $p=\frac{\lambda-1}{k}$.
\end{enumerate}
\end{proof}

\subsection{Perfect fluid spacetime satisfying $(\xi,\cdot)_{S}\cdot \mathcal{T}=0$}

Denote by
$$A=\lambda+\frac{k(\sigma-p)}{2}, \ \ B=k(\sigma+p),$$
$$b=\frac{r}{12}\left(=\frac{4\lambda+k(\sigma-3p)}{12}\right),\ \ c=-\frac{r}{6}\left(=-\frac{4\lambda+k(\sigma-3p)}{6}\right), \ \ d=\frac{1}{2}$$
and the curvature tensors can be written:
\begin{equation}\label{e33}
S(X,Y)=Ag(X,Y)+B\eta(X)\eta(Y)
\end{equation}
\begin{equation}\label{e35}
\mathcal{P}(X,Y)Z=R(X,Y)Z+b[g(Z,X)Y-g(Y,Z)X]
\end{equation}
\begin{equation}\label{e36}
\mathcal{C}(X,Y)Z=R(X,Y)Z+c[g(Z,X)Y-g(Y,Z)X]+$$$$+d[g(Z,X)QY-g(Y,Z)QX]+d[S(Z,X)Y-S(Y,Z)X]
\end{equation}
\begin{equation}\label{e37}
\mathcal{H}(X,Y)Z=R(X,Y)Z+$$$$+d[g(Z,X)QY-g(Y,Z)QX]+d[S(Z,X)Y-S(Y,Z)X]
\end{equation}
for any $X$, $Y$, $Z\in \chi(M)$.\\

The condition $(\xi,\cdot)_{S}\cdot \mathcal{T}=0$ is equivalent to
\begin{equation}\label{e777}
S(X,\mathcal{T}(Y,Z)W)\xi-S(\xi,\mathcal{T}(Y,Z)W)X+S(X,Y)\mathcal{T}(\xi,Z)W-$$
$$-S(\xi,Y)\mathcal{T}(X,Z)W+S(X,Z)\mathcal{T}(Y,\xi)W-S(\xi,Z)\mathcal{T}(Y,X)W+$$
$$+S(X,W)\mathcal{T}(Y,Z)\xi-S(\xi,W)\mathcal{T}(Y,Z)X=0,
\end{equation}
for any $X$, $Y$, $Z$, $W\in \chi(M)$.

\smallskip

Taking the inner product with $\xi$, the relation (\ref{e777}) becomes:
\begin{equation}\label{e38}
-S(X,\mathcal{T}(Y,Z)W)-S(\xi,\mathcal{T}(Y,Z)W)\eta(X)+$$$$+S(X,Y)\eta(\mathcal{T}(\xi,Z)W)-S(\xi,Y)\eta(\mathcal{T}(X,Z)W)+$$$$+
S(X,Z)\eta(\mathcal{T}(Y,\xi)W)-S(\xi,Z)\eta(\mathcal{T}(Y,X)W)+$$$$+S(X,W)\eta(\mathcal{T}(Y,Z)\xi)-S(\xi,W)\eta(\mathcal{T}(Y,Z)X)=0,
\end{equation}
for any $X$, $Y$, $Z$, $W\in \chi(M)$.

\begin{theorem}
Let $(M,g)$ be a general relativistic perfect fluid spacetime satisfying (\ref{e5}) with torse-forming vector field $\xi$.
\begin{enumerate}
  \item If $(\xi,\cdot)_{S}\cdot R=0$, then $p=\frac{\lambda}{k}$.
  \item If $(\xi,\cdot)_{S}\cdot \mathcal{P}=0$, then $p=\frac{\lambda}{k}$ or $p=\frac{4\lambda+k\sigma-12}{3k}$.
  \item If $(\xi,\cdot)_{S}\cdot \mathcal{C}=0$, then $p=\frac{\lambda}{k}$ or $p=\frac{2\lambda-k\sigma-6}{3k}$.
  \item If $(\xi,\cdot)_{S}\cdot \mathcal{H}=0$, then $p=\frac{\lambda}{k}$ or $p=\frac{\lambda-1}{k}$.
\end{enumerate}
\end{theorem}
\begin{proof}
\begin{enumerate}
  \item From (\ref{e38}) and (\ref{e33}) we obtain:
\begin{equation}
-Ag(X,R(Y,Z)W)+A[g(X,Z)g(Y,W)-g(X,Y)g(Z,W)]+$$$$+2A[\eta(X)\eta(Z)g(Y,W)-\eta(X)\eta(Y)g(Z,W)]+$$$$+B[\eta(Y)\eta(W)g(X,Z)-\eta(Z)\eta(W)g(X,Y)]=0,
\end{equation}
for any $X$, $Y$, $Z$, $W\in \chi(M)$. Taking $Z=W:=\xi$ we get:
\begin{equation}
(2A-B)[\eta(X)\eta(Y)+g(X,Y)]=0,
\end{equation}
for any $X$, $Y\in \chi(M)$. It follows $p=\frac{\lambda}{k}$.
  \item From (\ref{e38}), (\ref{e33}) and (\ref{e35}) we obtain:
\begin{equation}
-Ag(X,R(Y,Z)W)+A(1-2b)[g(X,Z)g(Y,W)-g(X,Y)g(Z,W)]+$$$$+2A(1-b)[\eta(X)\eta(Z)g(Y,W)-\eta(X)\eta(Y)g(Z,W)]+$$
$$+B(1-b)[\eta(Y)\eta(W)g(X,Z)-\eta(Z)\eta(W)g(X,Y)]=0,
\end{equation}
for any $X$, $Y$, $Z$, $W\in \chi(M)$. Taking $Z=W:=\xi$ we get:
\begin{equation}
(2A-B)(1-b)[\eta(X)\eta(Y)+g(X,Y)]=0,
\end{equation}
for any $X$, $Y\in \chi(M)$. It follows $p=\frac{\lambda}{k}$ or $p=\frac{4\lambda+k\sigma-12}{3k}$.
  \item From (\ref{e38}), (\ref{e33}) and (\ref{e36}) we obtain:
\begin{equation}
-Ag(X,R(Y,Z)W)+A(1-2c+dB-4dA)[g(X,Z)g(Y,W)-g(X,Y)g(Z,W)]+$$$$+A(2-2c+dB-4dA)[\eta(X)\eta(Z)g(Y,W)-\eta(X)\eta(Y)g(Z,W)]+$$
$$+B(1-c+dB-3dA)[\eta(Y)\eta(W)g(X,Z)-\eta(Z)\eta(W)g(X,Y)]=0,
\end{equation}
for any $X$, $Y$, $Z$, $W\in \chi(M)$. Taking $Z=W:=\xi$ we get:
\begin{equation}
(2A-B)[1-c-d(2A-B)][\eta(X)\eta(Y)+g(X,Y)]=0,
\end{equation}
for any $X$, $Y\in \chi(M)$. It follows $p=\frac{\lambda}{k}$ or $p=\frac{2\lambda-k\sigma-6}{3k}$.
  \item From (\ref{e38}), (\ref{e33}) and (\ref{e37}) we obtain:
\begin{equation}
-Ag(X,R(Y,Z)W)+A(1+dB-4dA)[g(X,Z)g(Y,W)-g(X,Y)g(Z,W)]+$$$$+A(2+dB-4dA)[\eta(X)\eta(Z)g(Y,W)-\eta(X)\eta(Y)g(Z,W)]+$$
$$+B(1+dB-3dA)[\eta(Y)\eta(W)g(X,Z)-\eta(Z)\eta(W)g(X,Y)]=0,
\end{equation}
for any $X$, $Y$, $Z$, $W\in \chi(M)$. Taking $Z=W:=\xi$ we get:
\begin{equation}
(2A-B)[1-d(2A-B)][\eta(X)\eta(Y)+g(X,Y)]=0,
\end{equation}
for any $X$, $Y\in \chi(M)$. It follows $p=\frac{\lambda}{k}$ or $p=\frac{\lambda-1}{k}$.
\end{enumerate}
\end{proof}

Remark the following facts:

\smallskip

1) If for a general relativistic perfect fluid spacetime $(M,g)$ satisfying (\ref{e5}) with torse-forming vector field $\xi$,
the conharmonic curvature tensor $\mathcal{H}$ satisfies
$(\xi,\cdot)_{\mathcal{H}}\cdot S=0$, then we have the vacuum case or the pressure is constant, but $(\xi,\cdot)_{R}\cdot S=0$ leads only to the vacuum case.

\smallskip

2) Under the same assumptions, the condition $(\xi,\cdot)_{S}\cdot R=0$ or $(\xi,\cdot)_{S}\cdot \mathcal{H}=0$ implies a constant pressure of the fluid.

\section{Solitons in a perfect fluid spacetime}

\subsection{$\eta$-Ricci solitons}

Consider the equation:
\begin{equation}\label{e11}
\mathcal{L}_{\xi}g+2S+2a g+2b\eta\otimes \eta=0,
\end{equation}
where $g$ is a pseudo-Riemannian metric, $S$ is the Ricci tensor, $\xi$ is a vector field, $\eta$ is a $1$-form and $a$ and $b$ are real constants. The data $(g,\xi,a,b)$ which satisfy
the equation (\ref{e11}) is said to be an \textit{$\eta$-Ricci soliton} in $M$ \cite{ch}; in particular, if $b=0$, $(g,\xi,a)$ is a \textit{Ricci soliton} \cite{ham} and it is called \textit{shrinking}, \textit{steady} or \textit{expanding} according as $a$ is negative, zero or positive, respectively \cite{chlu}.

Writing explicitly the Lie derivative $\mathcal{L}_{\xi}g$ we get $(\mathcal{L}_{\xi}g)(X,Y)=g(\nabla_X\xi,Y)+g(X,\nabla_Y\xi)$ and from (\ref{e11}) we obtain:
\begin{equation}\label{e51}
S(X,Y)=-ag(X,Y)-b\eta(X)\eta(Y)-\frac{1}{2}[g(\nabla_X\xi,Y)+g(X,\nabla_Y\xi)],
\end{equation}
for any $X$, $Y\in\chi(M)$.

Contracting (\ref{e51}) we get:
\begin{equation}
r=-a\dim(M)+b-div(\xi).
\end{equation}

Let $(M,g)$ be a general relativistic perfect fluid spacetime and $(g,\xi,a,b)$ be an $\eta$-Ricci soliton in $M$. From (\ref{e52}) and (\ref{e51}) we obtain:
\begin{equation}\label{e4}
\left[\lambda+\frac{k(\sigma-p)}{2}+a\right]g(X,Y)+[k(\sigma+p)+b]\eta(X)\eta(Y)+$$$$+\frac{1}{2}[g(\nabla_X\xi,Y)+g(X,\nabla_Y\xi)]=0,
\end{equation}
for any $X$, $Y\in\chi(M)$.

Consider $\{E_i\}_{1\leq i \leq 4}$ an orthonormal frame field and let $\xi=\sum_{i=1}^{4}\xi^iE_i$. We have shown in the previous section that $\sum_{i=1}^4\varepsilon_{ii}(\xi^i)^2=-1$ and $\eta(E_i)=\varepsilon_{ii}\xi^i$.

Multiplying (\ref{e4}) by $\varepsilon_{ii}$ and summing over $i$ for $X=Y:=E_i$, we get:
\begin{equation}\label{e6}
4a-b=-4\lambda-k(\sigma-3p)-div(\xi).
\end{equation}

Writing (\ref{e4}) for $X=Y:=\xi$, we obtain:
\begin{equation}\label{e7}
a-b=-\lambda+\frac{k(\sigma+3p)}{2}.
\end{equation}

Therefore:
\begin{equation}
\left\{
  \begin{array}{ll}
    a=-\lambda-\frac{k(\sigma-p)}{2}-\frac{div(\xi)}{3} \\
    b=-k(\sigma+p)-\frac{div(\xi)}{3}
  \end{array}
\right..
\end{equation}

\begin{theorem}
Let $(M,g)$ be a $4$-dimensional pseudo-Riemannian manifold and
let $\eta$ be the $g$-dual $1$-form of the gradient vector field $\xi:=grad(f)$ with $g(\xi,\xi)=-1$. If (\ref{e11}) defines an $\eta$-Ricci soliton in $M$, then the Poisson equation satisfied by $f$ is:
\begin{equation}
\Delta(f)=-3[b+k(\sigma+p)].
\end{equation}
\end{theorem}

\begin{remark}
If $b=0$ in (\ref{e11}), we obtain the Ricci soliton with $a=-\lambda+\frac{k(\sigma+3p)}{2}$ which is steady if $p=\frac{2}{3}(\frac{\lambda}{k})-\frac{\sigma}{3}$, expanding if $p>\frac{2}{3}(\frac{\lambda}{k})-\frac{\sigma}{3}$ and shrinking if $p<\frac{2}{3}(\frac{\lambda}{k})-\frac{\sigma}{3}$. In these cases, $div(\xi)=-3k(\sigma+p)$. From Plebanski energy conditions for perfect fluids we deduce that $\sigma\geq\max\{-\frac{\lambda}{k},\frac{\lambda}{2k}\}$ for the steady case, $\sigma>\frac{\lambda}{2k}$ and $\sigma>-\frac{\lambda}{k}$ for the expanding and shrinking case, respectively.
\end{remark}

\begin{example}\label{ex1}
An $\eta$-Ricci soliton $(g,\xi,a,b)$ in a radiation fluid is given by
$$\left\{
  \begin{array}{ll}
    a=-\lambda-kp-\frac{div(\xi)}{3} \\
    b=-4kp-\frac{div(\xi)}{3}
  \end{array}
\right..$$
\end{example}

\subsection{$\eta$-Einstein solitons}

Consider the equation:
\begin{equation}\label{e41}
\mathcal{L}_{\xi}g+2S+(2a-r) g+2b\eta\otimes \eta=0,
\end{equation}
where $g$ is a pseudo-Riemannian metric, $S$ is the Ricci tensor, $r$ is the scalar curvature, $\xi$ is a vector field, $\eta$ is a $1$-form and $a$ and $b$ are real constants. The data $(g,\xi,a,b)$ which satisfy
the equation (\ref{e41}) is said to be an \textit{$\eta$-Einstein soliton} in $M$; in particular, if $b=0$, $(g,\xi,a)$ is an \textit{Einstein soliton} \cite{ca}.

Writing explicitly the Lie derivative $\mathcal{L}_{\xi}g$ we get $(\mathcal{L}_{\xi}g)(X,Y)=g(\nabla_X\xi,Y)+g(X,\nabla_Y\xi)$ and from (\ref{e41}) we obtain:
\begin{equation}\label{e42}
S(X,Y)=-\left(a-\frac{r}{2}\right)g(X,Y)-b\eta(X)\eta(Y)-\frac{1}{2}[g(\nabla_X\xi,Y)+g(X,\nabla_Y\xi)],
\end{equation}
for any $X$, $Y\in\chi(M)$.

Contracting (\ref{e42}) we get:
\begin{equation}
\frac{2-\dim(M)}{2}r=-a\dim(M)+b-div(\xi).
\end{equation}

Let $(M,g)$ be a general relativistic perfect fluid spacetime and $(g,\xi,a,b)$ be an $\eta$-Einstein soliton in $M$. From (\ref{e8}), (\ref{e52}) and (\ref{e42}) we obtain:
\begin{equation}\label{e43}
(\lambda-kp-a)g(X,Y)-[k(\sigma+p)+b]\eta(X)\eta(Y)-\frac{1}{2}[g(\nabla_X\xi,Y)+g(X,\nabla_Y\xi)]=0,
\end{equation}
for any $X$, $Y\in\chi(M)$.

Consider $\{E_i\}_{1\leq i \leq 4}$ an orthonormal frame field and let $\xi=\sum_{i=1}^{4}\xi^iE_i$. We have shown in the previous section that $\sum_{i=1}^4\varepsilon_{ii}(\xi^i)^2=-1$ and $\eta(E_i)=\varepsilon_{ii}\xi^i$.

Multiplying (\ref{e43}) by $\varepsilon_{ii}$ and summing over $i$ for $X=Y:=E_i$, we get:
\begin{equation}\label{e6}
4a-b=4\lambda+k(\sigma-3p)-div(\xi).
\end{equation}

Writing (\ref{e43}) for $X=Y:=\xi$, we obtain:
\begin{equation}\label{e7}
a-b=\lambda+k\sigma.
\end{equation}

Therefore:
\begin{equation}
\left\{
  \begin{array}{ll}
    a=\lambda-kp-\frac{div(\xi)}{3} \\
    b=-k(\sigma+p)-\frac{div(\xi)}{3}
  \end{array}
\right..
\end{equation}

\begin{theorem}
Let $(M,g)$ be a $4$-dimensional pseudo-Riemannian manifold and
let $\eta$ be the $g$-dual $1$-form of the gradient vector field $\xi:=grad(f)$ with $g(\xi,\xi)=-1$. If (\ref{e41}) defines an $\eta$-Einstein soliton in $M$, then the Poisson equation satisfied by $f$ is:
\begin{equation}
\Delta(f)=-3[b+k(\sigma+p)].
\end{equation}
\end{theorem}

\begin{remark}
If $b=0$ in (\ref{e41}), we obtain the Einstein soliton with $a=\lambda+k\sigma$ which is steady if
$\sigma=-\frac{\lambda}{k}$, expanding if $\sigma>-\frac{\lambda}{k}$ and shrinking if $\sigma<-\frac{\lambda}{k}$. In these cases, $div(\xi)=-3k(\sigma+p)$.
\end{remark}

\begin{example}\label{ex2}
An $\eta$-Einstein soliton $(g,\xi,a,b)$ in a radiation fluid is given by
$$\left\{
  \begin{array}{ll}
    a=\lambda-kp-\frac{div(\xi)}{3} \\
    b=-4kp-\frac{div(\xi)}{3}
  \end{array}
\right..$$
\end{example}

\medskip

Remark the following facts:

\smallskip

1) From Example \ref{ex1} we deduce that the Ricci soliton in a radiation fluid is steady if $p=\frac{\lambda}{3k}$,
expanding if $p>\frac{\lambda}{3k}$ and shrinking if $p<\frac{\lambda}{3k}$.

\smallskip

2) In a general relativistic perfect fluid spacetime, if the vector field $\xi$ is torse-forming with $\nabla_X\xi=s[X+\eta(X)\xi]$, for any $X\in \chi(M)$ and $s$ a nonzero real number, then $div(\xi)=3s$. In this case, the existence of a Ricci soliton given by (\ref{e11}) for $b=0$, from Plebanski energy conditions implies $-2k\sigma\leq s<0$ (precisely, $s=-k(\sigma+p)$).

\smallskip

3) If the vector field $\xi$ is conformal Killing, i.e. $L_{\xi}g=sg$ with $s$ a nonzero real number, then the existence of a Ricci soliton given by (\ref{e11}) for $b=0$, implies the vacuum case. Moreover, the soliton is steady if $p=\frac{\lambda}{k}+\frac{s}{2k}$, expanding if $p>\frac{\lambda}{k}+\frac{s}{2k}$ and shrinking if $p<\frac{\lambda}{k}+\frac{s}{2k}$.

\smallskip

4) The existence of a steady Ricci soliton or a (weakly) pseudo Ricci symmetric Ricci tensor field in a general relativistic perfect fluid spacetime with $p$ and $\sigma$ constant imply the same pressure of the fluid, a surprisingly fact being that the $1$-form $\alpha$, arbitrary chosen, does not effectively appear.

\small{

\bigskip

\bigskip
\bigskip
\bigskip
\bigskip
\bigskip

\textit{Adara M. Blaga}

\textit{Department of Mathematics}

\textit{West University of Timi\c{s}oara}

\textit{Bld. V. P\^{a}rvan nr. 4, 300223, Timi\c{s}oara, Rom\^{a}nia}

\textit{adarablaga@yahoo.com}
}


\begin{thebibliography}{99}
\bibitem{ad} {Adati, A. and Miyazawa, T.}: \textit{On a Riemannian Space with recurrent conformal curvature}, Tensor N.S. \textbf{18}, 355-342 (1967).

\bibitem{ah} {Ahsan, Z. and Ali, M.}: \textit{Curvature Tensor for the Spacetime of General Relativity}, Int. J. Geom. Methods Mod. Phys. \textbf{14}, 13 pages, (2017).

\bibitem{b12}Blaga, A. M.: \emph{A note on almost $\eta$-Ricci solitons in Euclidean hypersurfaces}. Serdica Math. J. \textbf{43}(3-4), 361-368 (2017).

\bibitem{bla}Blaga, A. M.: \emph{$\eta$-Ricci solitons on Lorentzian para-Sasakian manifolds}. Filomat \textbf{30}(2), 489-496 (2016).

\bibitem{bl}Blaga, A. M.: \emph{$\eta$-Ricci solitons on para-Kenmotsu manifolds}. Balkan J. Geom. Appl. \textbf{20}(1), 1-13 (2015).

\bibitem{blaga}Blaga, A. M.: \emph{On gradient $\eta$-Einstein solitons}. Kragujevak J. Math. \textbf{42}(2), 229-237 (2018).


\bibitem{b3}Blaga, A. M.: \emph{On solitons in statistical geometry}. Int. J. Appl. Math. Stat. \textbf{58}(4), (2019).

\bibitem{blag}Blaga, A. M.: \emph{On warped product gradient $\eta$-Ricci solitons}. Filomat \textbf{31}(18), 5791-5801 (2017).

\bibitem{b7}Blaga, A. M.: \emph{Remarks on almost $\eta$-solitons}. Matematicki Vesnik \textbf{71}(3), 244-249 (2019).

\bibitem{b1}Blaga, A. M.: \emph{Solutions of some types of soliton equations in $\mathbb{R}^3$}. Filomat \textbf{33}(4), 1159-1162 (2019).

\bibitem{b5}Blaga, A. M.: \emph{Some geometrical aspects of Einstein, Ricci and Yamabe solitons}. J. Geom. Sym. Phys. \textbf{52}, 17-26 (2019).


\bibitem{b2}{Blaga, A. M., Crasmareanu, M. C.}: \emph{Inequalities for gradient Einstein and Ricci solitons}. Facta Univ. Math. Inform. \textbf{35}(2), 351-356 (2020).

\bibitem{blcr}{Blaga, A. M., Crasmareanu, M. C.}: \emph{Torse-forming $\eta$-Ricci solitons in almost paracontact $\eta$-Einstein geometry}. Filomat \textbf{31}(2), 499-504 (2017).

\bibitem{b8}{Blaga, A. M., Perkta\c s, S. Y.}: \emph{Remarks on almost $\eta$-Ricci solitons in $\varepsilon$-para Sasakian manifolds}. Commun. Fac. Sci. Univ. Ank. Ser. A1 Math. Stat. \textbf{68}(2), 1621-1628 (2019).

\bibitem{b9}{Blaga, A. M., Perkta\c s, S. Y., Acet, B. E., Erdogan, F. E.}: \emph{$\eta$-Ricci solitons in $\varepsilon$-almost paracontact metric manifolds}. Glasnik Matematicki \textbf{53}(1), 377-410 (2018).





\bibitem{ca} {Catino, G., Mazzieri, L.}: \textit{Gradient Einstein solitons}, Nonlinear Anal. \textbf{132}, 66-94 (2016).

\bibitem{mai} {Chaki, M. C., Maity, R. K.}: \textit{On quasi Einstein manifolds}, Publ. Math. Debrecen \textbf{57}, 297-306 (2000).

\bibitem{cha} {Chaki, M. C., Ray, S.}: \textit{Space-times with covariant-constant energy momentum tensor}, Internat. J. Theoret. Phys. \textbf{35}, 1027-1032 (1996).

\bibitem{ch} {Cho, J. T., Kimura, M.}: \textit{Ricci solitons and real hypersurfaces in a complex space form}, Tohoku Math. J. \textbf{61}(2), 205-212 (2009).

\bibitem{chlu} {Chow, B., Lu, P., Ni, L.}: \textit{Hamilton's Ricci Flow}, Graduate Studies in Mathematics \textbf{77}, AMS, Providence, RI, USA, 2006.

\bibitem{cr} {Crasmareanu, M.}: \textit{Parallel tensors and Ricci solitons in $N(k)$-quasi Einstein manifolds}, Indian J. Pure Appl. Math. \textbf{43}(4), 359-369 (2012).


\bibitem{deg} {De, U. C., Gazi, A. K.}: \textit{On pseudo Ricci symmetric manifolds}, An. \c Stiin\c t. Univ. Al. I. Cuza Ia\c si. Mat. (N.S.), tom \textbf{LVIII}, f. 1, 209-222 (2012).

\bibitem{degh} {De, U. C., Ghosh, G. C.}: \textit{On quasi-Einstein and special quasi-Einstein manifolds}, Proc.
of the Int. Conf. of Mathematics and its Applications, Kuwait University, April 5-7, 178-191 (2004).

\bibitem{degho} {De, U. C., Ghosh, G. C.}: \textit{On quasi-Einstein manifolds}, Period. Math. Hungar. \textbf{48}(1-2), 223-231 (2004).

\bibitem{de} {Deszcz, R., Hotlos, M., Senturk, Z.}: \textit{On curvature properties of quasi-Einstein hypersurfaces
in semi-Euclidean spaces}, Soochow J. Math. \textbf{27}, 375-389 (2001).

\bibitem{ham} {Hamilton, R. S.}: \textit{The Ricci flow on surfaces}, Math. and general relativity (Santa Cruz, CA, 1986), 237-262, Contemp. Math. \textbf{71}, AMS, (1988).

\bibitem{matsu} {Matsumoto, K.}: \textit{On Lorentzian paracontact manifolds}, Bull. of Yamagata Univ. Nat. Sci. \textbf{12}, 151-156 (1989).

\bibitem{ne} {Neill, O'.}: \textit{Semi-Riemannian Geometry with Applications to Relativity}, Pure and Applied Math. \textbf{103}, Academic Press, New York, (1983).

\bibitem{po} {Pokhariyal, G. P., Mishra, R. S.}: \textit{Curvature tensors and their relativistics significance}, Yokohama Math. J. \textbf{18}, 105-108 (1970).

\bibitem{shaik} {Shaikh, A. A.}: \textit{On Lorentzian almost paracontact manifolds with a structure of the concircular type}, Kyungpook J. Math. \textbf{43}, 305-314 (2003).

\bibitem{st} {Stephani, H.}: \textit{General Relativity-An Introduction to the Theory of Gravitational Field}, Cambridge University Press, Cambridge, (1982).

\bibitem{yan} {Yano, K.}: \textit{Concircular geometry I. Concircular transformations}, Proc. Imp. Acad. Tokyo \textbf{16}, 195-200 (1940).

\bibitem{ya} {Yano, K.}: \textit{On torse forming direction in a Riemannian space}, Proc. Imp. Acad. Tokyo \textbf{20}, 340-345 (1944).

\end{thebibliography}
\end{document}